\documentclass[12pt,leqno,fleqn]{amsart}  
\usepackage{amsmath,amstext,amsthm,amssymb,amsxtra}
\usepackage[top=1.5in, bottom=1.5in, left=1.25in, right=1.25in]	{geometry}
\usepackage{txfonts} 
\usepackage[T1]{fontenc}
\usepackage{lmodern}

 \usepackage{euler}   

\usepackage{mathtools}
\mathtoolsset{showonlyrefs,showmanualtags}

\usepackage{hyperref} 
\hypersetup{
    colorlinks=true,       
    linkcolor=blue,          
    citecolor=magenta,        
    filecolor=magenta,      
    urlcolor=cyan           
}

\usepackage[msc-links]{amsrefs}

\theoremstyle{plain} 
\newtheorem{theorem}{Theorem}[section] 
\newtheorem{lemma}[theorem]{Lemma} 
 
\newtheorem{corollary}[theorem]{Corollary}

\theoremstyle{definition}

\theoremstyle{remark}
\newtheorem{remark}[equation]{Remark}

\numberwithin{equation}{section}

\title[Averages along Primes] {Averages Along the Primes: Improving and Sparse Bounds}

\author[Han] {Rui Han} 
\thanks{RH: Research supported in part by grant National Science Foundation grant DMS-1800689.}
\address{ School of Mathematics, Georgia Institute of Technology, Atlanta GA 30332, USA}
\email {rui.han@math.gatech.edu}

\author[Krause] {Ben Krause}
\address{Department of Mathematics, Princeton University, Princeton NJ 08544}

\author[Lacey] {Michael T. Lacey} 
\thanks{MTL: Research supported in part by grant National Science Foundation grant DMS-1600693, and by Australian Research Council grant DP160100153. }
\address{ School of Mathematics, Georgia Institute of Technology, Atlanta GA 30332, USA}
\email {lacey@math.gatech.edu}

\author[Yang]{ Fan Yang}
\thanks{FY: Research supported in part by grant National Science Foundation grant DMS-1600693 and AMS-Simons Travel grant 2019-2021.}
\address{ School of Mathematics, Georgia Institute of Technology, Atlanta GA 30332, USA}
\email {fan.yang@math.gatech.edu}

\makeatletter
\DeclareRobustCommand\widecheck[1]{{\mathpalette\@widecheck{#1}}}
\def\@widecheck#1#2{%
    \setbox\z@\hbox{\m@th$#1#2$}%
    \setbox\tw@\hbox{\m@th$#1%
       \widehat{%
          \vrule\@width\z@\@height\ht\z@
          \vrule\@height\z@\@width\wd\z@}$}%
    \dp\tw@-\ht\z@
    \@tempdima\ht\z@ \advance\@tempdima2\ht\tw@ \divide\@tempdima\thr@@
    \setbox\tw@\hbox{%
       \raise\@tempdima\hbox{\scalebox{1}[-1]{\lower\@tempdima\box
\tw@}}}%
    {\ooalign{\box\tw@ \cr \box\z@}}}
\makeatother

\begin{document}

\begin{abstract}
Consider  averages along the prime integers $ \mathbb P $ given by 
\begin{equation*}
\mathcal{A}_N f (x)  = N ^{-1}  \sum_{ p \in \mathbb P \;:\; p\leq N}  (\log p) f (x-p). 
\end{equation*}
These averages satisfy a uniform scale-free $ \ell ^{p}$-improving estimate. For all $ 1< p < 2$, there is a constant $ C_p$ so that for all integer $ N$ and functions $ f$ supported on $ [0,N]$, there holds 
\begin{equation*}
N ^{-1/p' }\lVert \mathcal{A}_N f\rVert _{\ell^{p'}} \leq C_p  N ^{- 1/p} \lVert f\rVert _{\ell^p}.
\end{equation*}
The maximal function $ \mathcal{A}^{\ast}  f =\sup _{N} \lvert \mathcal{A}_N f \rvert$ satisfies $ (p,p)$ sparse bounds for all $ 1< p < 2$. The latter are  the natural variants of the scale-free bounds.   
As a corollary,  $  \mathcal{A}^{\ast}  $ is bounded on $ \ell ^{p} (w)$, for all weights $ w$ in the Muckenhoupt $A_p$ class. No prior weighted inequalities for $ \mathcal{A}^{\ast} $ were known. 
\end{abstract}

	\maketitle  
\tableofcontents 

\section{Introduction} 
	
Let $ \mathbb P = \{3, 5, 7 ,\dotsc, \}$ be the odd primes and define the logarithmically weighted 
 averages along the primes by 
\begin{align*}
\mathcal{A}_N f (x) & = N ^{-1}  \sum_{ p \in \mathbb P \;:\; p\leq N}  (\log p) f (x-p), 
\end{align*}
We prove scale-free $ \ell ^{p}$ improving bounds for these averages, and sparse bounds for the associated maximal function 
\begin{equation} \label{e:A*}
\mathcal{A} ^{\ast} f = \sup _{N}  \lvert \mathcal{A}_N f \rvert . 
\end{equation}

For a function $f$ on $\mathbb Z $, and an interval $I\subset \mathbb Z $, define 
\begin{align}\label{def:p_norm}
\langle f\rangle_{I,p}:=\left(\frac{1}{|I|} \sum_{x\in I} |f(x)|^p\right)^{1/p}
\end{align}
to be the normalized $\ell^p$ norm on $I$.
Throughout the paper, if $I=[a,b]\cap \mathbb Z$, with $a,b\in \mathbb Z$, is an interval on $\mathbb Z$,  let $2I=[2a-b-1, b]\cap \mathbb Z$ be the doubled interval (on the left-hand side), let $3I=[2a-b-1,2b-a+1]$ be the tripled interval which has the same center as $I$. 
 
We prove that the averages along the primes improve integrability, uniformly over all scales. 

\begin{theorem}\label{t:fixed}  
 For $ 1< p < \infty $, there is a constant $ C_p$ so that for all integers $ N$, 
and   interval $ I$ of length $ N$, there holds for all functions $ f$, 
\begin{equation}\label{e:fixed}
\langle  \mathcal{A}_N f  \rangle _{I,p'} \leq C_p \langle f \rangle _{2I, p} , 
\end{equation}
where $ p' = \frac{p} {p-1}$.  
\end{theorem}

We turn to the sparse inequalities.  They are the natural 
extensions of the $ \ell ^{p}$ improving inequalities above for the maximal function \eqref{e:A*}.  
 We say that a sublinear operator $ B$  has sparse type $ (r,s)$, for $ 1< r, s < \infty $ if there is a constant $ C$ so that 
for all finitely supported functions $ f,g$ there are a sparse collection of intervals $ \mathcal S $ so that 
\begin{equation}  \label{e:B}
\lvert  ( Bf,g )\rvert \leq C \sum_{I\in \mathcal S} \langle f \rangle _{2I,r} \langle g \rangle _{I,s} \lvert  I\rvert,
\end{equation}
where $(f,g)$ is the standard inner product on $\ell^2(\mathbb Z)$.
A collection of intervals $ \mathcal S$ is said to be \emph{sparse} if there are subsets $ E_I \subset I$ for $ I\in \mathcal S$ which are 
pairwise disjoint, and satisfy $ \lvert  E_I\rvert > \frac{1} {10} \lvert  I\rvert  $.

\begin{theorem}\label{t:sparse}
The maximal operator $ \mathcal{A}^{\ast} $    is of sparse type $ (r,s)$, for all $ 1< r,s < 2$. 

\end{theorem}

This statement is much stronger than just asserting that $ \mathcal{A}^{\ast} $ is bounded on $ \ell ^{p}$, for all $ 1< p < \infty $. It implies 
for instance these weighted inequalities, which  match the classical result of Muckenhoupt for the ordinary maximal function.  
(Although the quantitative estimates of the norm will not match.)

\begin{corollary}\label{c:wtd} For any $ 1< p < \infty $, and any weight $ w$  in the Muckenhoupt class $ A_p$, we have that $ \mathcal{A}^{\ast} $ 
is a bounded operator on  $ \ell ^{p} (w)$. 
\end{corollary}

We remark that for the simple averages along the primes, one can check that for non-negative $ f$
\begin{equation*}
\sup _{N}  \frac{\log N} N \sum_{ p \in \mathbb P \;:\; p\leq N} f (x-p ) \lesssim \mathcal{A}^{\ast} f .
\end{equation*}
Therefore, the sparse bounds hold for the maximal function on the left.  
Our argument for the  fixed scale inequalities \eqref{e:fixed} requires the logarithmic averages. 

\bigskip 
Following Bourgain's work on arithmetic ergodic theorems \cite{MR937581}, Wierdl \cite{MR995574} showed that $ \mathcal{A}^{\ast} $ 
is bounded on $ \ell ^{p}$ for all $ 1< p < \infty $. At the time, this was the first arithmetic example for which this fact was known for all $ 1< p < 2$. 
Bourgain's work \cite{MR1019960} gave a comprehensive approach to the $ \ell ^{p}$ theory of arithmetic averages.   
The subject continues to be under development, with important contributions by \cites{MR2188130,MR3681393,2015arXiv151207518M}. 
We point to the work of Mirek-Trojan and Trojan \cites{2019arXiv190704753T,MR3370012} also focused on the primes. The methods therein 
are different from those of this paper.  

Our subject, developing the $ \ell ^{p}$-improving properties and sparse bounds started with \cite{MR3892403}, 
and continued in  \cites{180409260H,180409845L}.  It now encompasses the discrete spherical maximal operators \cites{2018arXiv181002240K,181012344K,180509925,180906468}, as well as the square integers \cite{2019arXiv190705734H}. 

\medskip 
We use  the High Low Method  \cites{2019arXiv190705734H,181002240,I}.  This depends upon efficient use of $ \ell ^{2}$-methods, followed by a fine analysis of 
certain $ \ell ^{1}$-type expressions. The latter are frequently the most intricate part. In this argument, they depend upon a relatively 
accessible property of Ramanujan sums, Lemma~\ref{l:Ram}.  Our argument is new, even if one is only interested in the $ \ell ^{p} \to \ell ^{p}$ 
bounds for $\mathcal{A}^{\ast }$.


\section{Preliminaries} 

Throughout, let $\phi(q)$ be the Euler totient function, let $\mu(q)$ be the M\"obius function.
The following estimate for $\phi(q)$ is well known:
\begin{align}\label{e:phi}
\phi(q)\gtrsim_{\varepsilon} q^{1-\varepsilon}.
\end{align}
We count primes in the standard logarithmic fashion.   
Put  
\begin{align}\label{e:Pi}
\vartheta (N) &= \sum_{p\in \mathbb P \;:\; p\leq N} \log P.
\end{align}
By the prime number theorem
\begin{align}\label{e:Pi_est}
\left| \frac{\vartheta (N)-N}{N}\right| \leq C e^{-c\sqrt{log N}},
\end{align}
holds for some constant $c, C>0$. This obviously implies $\vartheta (N)\sim N$.




We now redefine the averaging operators $ A_N$, by setting 
\begin{equation}\label{e:redefine}
\mathcal{A}_N f (x) = \vartheta (N) ^{-1} \sum_{p\in \mathbb P \;:\; p\leq N} (\log p)\, f (x-p)
\end{equation}
As this is a positive operator, there is no harm in this new definition.

The Fourier transform of a measure $ \sigma $ on $ \mathbb Z $ is given by 
\begin{equation*}
\widehat \sigma  (\xi ) = \sum_{x\in \mathbb Z } \sigma (x) e (x \xi),   
\end{equation*}
where $e( \zeta )=e^{2\pi i \zeta }$ throughout. The inverse Fourier transform is denoted $\widecheck   \eta $. 
Occasionally, we may also denote the Fourier transform by $\mathcal{F}$, and inverse Fourier transform by $\mathcal{F}^{-1}$.

We further set $e_q( \zeta ) = e^{2\pi i \zeta /q}$. 
Recall that Ramanujan sums are defined by 
\begin{equation}\label{e:cq}
c _{q} (n) = \sum_{a \in \mathbb A _q} e_q( an/q) , 
\end{equation}
where $ \mathbb A _q = \{ 1\leq a < q \;:\; (a,q)=1\}$ is the multiplicative group associated to $ q$.  
Define by convention that $c_1(n)\equiv 1$.





A finer property of Ramanujan sums is recalled in Lemma~\ref{l:Ram} below.

\section{Approximating Multipliers} 

We define the approximating multipliers.
Let $ \mathbf 1 _{[-1/8,1/8]} \leq \eta \leq \mathbf 1_{[-1/4,1/4] } $ be a Schwartz function. 
For an integer $s$, let $\eta_s (\xi ) = \eta( 8^s \xi) $.  Define the Fourier transform of the usual averages by 
\begin{equation} \label{e:gamma}
 \widehat {\gamma_N} = \frac 1 N \sum_{n=1}^N \widehat {\delta_{n}} . 
\end{equation} 
The building blocks of the approximating multipliers are 
\begin{equation} \label{e:LqN}
\begin{split}
\widehat {L _{1,N}}  (\xi ) &= \widehat {\gamma_N} (\xi ) \eta _1 ( \xi )
\\
\widehat {L _{q,N}}  ( \xi) &= \frac { \mu(q)} { \phi (q)} \sum_{a\in \mathbb A_q}  \widehat {\gamma_N}  \cdot \eta_s  ( \xi -a/q)  , \qquad  2^s \leq q < 2 ^{s+1}, \,\, s\geq 1. 
\end{split}
\end{equation} 
Throughout, $q$ and $s$ have the relationship above, although this will be suppressed in the notation.
(This is a useful convention in the application of the multi-frequency maximal 
function inequality in the proof of the sparse bounds, see \eqref{e:multi}.) 
 
\begin{theorem}\label{t:approx}  Let $ A, N>10$ be  integers. 
If $K \lesssim (\log N)^A$, there holds 
\begin{equation}   \label{e:Approx} 
\widehat {\mathcal{A}_N}  = \sum_{1\leq q \leq K } \widehat{L _{q,N}}   +  r_{A,N,K}, 
\end{equation}
where $ \| r_{A,N,K} \|_{L^\infty} \lesssim_A  K^{-1 + 1/A} $.  

\end{theorem}

This is a consequence of  standard facts in the number theory literature, and is very similar to how these 
facts are used in \cite{MR995574}. We recall them here. 

\begin{lemma}\label{l:approx} For positive integers $ B $, there is an integer $ N_B$ so that for all  $ N> N_B$
\begin{enumerate}
\item   If $ \lvert  \xi \rvert < \frac{ (\log N) ^{B}} {N} $, then 
\begin{equation}\label{e:0Major}
 \widehat {\mathcal{A}_N} (\xi ) = \widehat {\gamma_N} (\xi ) + O (e ^{-c \sqrt {\log N}}).  
\end{equation}

\item If $ \lvert  \xi -a/q\rvert < \frac{ (\log N) ^{B}} {N} $ for $ (a,q)=1$ and $ 1 < q < (\log N) ^{B}$, then 
 \begin{equation}\label{e:Major}
 \widehat {\mathcal{A}_N} (\xi ) =  \frac{\mu (q)} {\phi (q)}\widehat {\gamma_N}(\xi-\frac{a}{q} ) + O (e ^{-c \sqrt {\log N}}).  
\end{equation}

\item If $ \xi $ does not meet any of the hypotheses of the prior two conditions, then 
\begin{equation}\label{e:minor}
 \widehat {\mathcal{A}_N} (\xi )   = O \left( ( \log N) ^{4-\frac{B}{2}} \right).
\end{equation}

\item  The following holds for $\lvert \xi\rvert\leq 1/2$
\begin{equation}\label{e:gamma_est}
\lvert  \widehat {\gamma_N} ( \xi  )\rvert  \lesssim  \min \{ 1,    (N \lvert  \xi \rvert ) ^{-1} \}.
\end{equation}

\end{enumerate}
\end{lemma}

The points (1), (2) and (3) above are in \cite{MR628618}*{Lemma 3.1 \& Thm. 3.1}, while the last point is  well known.

\begin{proof}[Proof of Theorem~\ref{t:approx}]  
We note that by construction, the multipliers $ \{\widehat{L _{q,N}} \;:\; 2 ^{s} \leq q < 2 ^{s+1}\}$ are supported on disjoint intervals 
around the rationals $ a/q$, with $ a\in \mathbb A _q$, and $ 2 ^{s} \leq q < 2 ^{s+1}$.   From this, it follows from \eqref{e:phi} that 
\begin{equation} \label{e:s}
\Bigl\lVert  \sum_{2 ^{s} \leq q < 2 ^{s+1}} \widehat{L _{q,N}} \Bigr\rVert _{L^\infty } \leq \max _{2 ^{s} \leq q< 2 ^{s+1}} \phi (q) ^{-1} \lesssim
2 ^{-s (1-1/A)}.  
\end{equation}
Above, $ A$ is the integer in Theorem~\ref{t:approx}. 

It suffices to argue that for $ B = 2A+8$
\begin{equation}\label{e:usual}
\widehat {\mathcal{A}_N} = \sum_{1\leq q \leq (\log N) ^{B}} \widehat{L _{q,N}} + O ( \log N ) ^{-A},  
\end{equation}
because we can use \eqref{e:s} to complete the proof of \eqref{e:Approx}.  

We note that the intervals of $ \xi $ that appear in the conditions 1 and 2 of Lemma~\ref{l:approx} are pairwise disjoint. 
Let us assume that $ \xi $ meets the condition 2, so $  \lvert  \xi -a/q\rvert < \frac{ (\log N) ^{B}} {N} $ for $ (a,q)=1$ and $ 1 < q < (\log N) ^{B}$. 
To prove \eqref{e:Approx} in this case, we need to see that,
\begin{align*}
\eta_s(\xi-b/q)=
\begin{cases} 
0,\qquad \text{if } \mathbb A_q \ni b\neq a\\
1,\qquad \text{if } b=a
\end{cases}
\end{align*}
Hence
\begin{align*}
\widehat{L _{q,N}} (\xi )=\frac{\mu (q)} {\phi (q)} \widehat {\gamma_N}(\xi -a/q),
\end{align*}
and furthermore by \eqref{e:Major},
\begin{equation}\label{e:A-Lq}
\lvert \widehat{\mathcal{A}_N}(\xi)-\widehat{L_{q,N}}(\xi) \rvert\leq e^{-c\sqrt{\log N}}.
\end{equation}
We also need to see that all the other $ L _{q',N} (\xi )$ are small. Indeed, for $ 1< q'\neq q \leq (\log N) ^{B}$, and $ a'\in \mathbb A _{q'}$, 
we have $ \lvert  \xi -a'/q'\rvert \geq \frac{ (\log N) ^{B}} {N}$. 
Hence, by \eqref{e:gamma_est}, we have 
\begin{align*}
\lvert  \widehat{L _{q',N}} (\xi )\rvert \lesssim  \phi (q') ^{-1}  (\log N) ^{-B}. 
\end{align*}
Similarly, we have $|\xi|\geq \frac{ (\log N) ^{B}} {N}$, hence
\begin{align*}
\lvert  \widehat{L _{1,N}} (\xi )\rvert \lesssim (\log N) ^{-B}. 
\end{align*}
Summing the estimates for $\widehat{L_{q',N}}$ over $1 \leq q'\neq q \leq (\log N) ^{B}$ and using \eqref{e:phi}, we have
\begin{equation}\label{e:MIN}
\sum_{1\leq q'\neq q\leq (\log N)^B} \lvert \widehat{L_{q', N}}(\xi) \rvert \lesssim (\log N)^{-A}.
\end{equation}
Putting \eqref{e:A-Lq} and \eqref{e:MIN} together, we have verified \eqref{e:usual} in this case.  
If $\xi$ meets condition 1 of Lemma~\ref{l:approx}, the proof is completely analogous.

We now assume that $ \xi $ does not meet the first or second condition of Lemma~\ref{l:approx}.  Then, \eqref{e:minor} holds. 
And, similar to \eqref{e:MIN}, we have
\begin{equation*}
\sum_{1\leq q \leq (\log N) ^{B}} \lvert  \widehat{L _{q,N}} (\xi )\rvert 
\lesssim \sum_{1\leq q \leq (\log N) ^{B}} \phi (q) ^{-1}  (\log N) ^{-B} \lesssim (\log N) ^{-B+1}.  
\end{equation*}
Combining \eqref{e:minor} with \eqref{e:MIN}, we have completed the proof of \eqref{e:usual}.  
\end{proof}

The building blocks of the approximating multipliers have explicit inverse Fourier transforms. 

\begin{lemma}\label{l:invert} With the notation of \eqref{e:LqN}, there holds 
\begin{equation} \label{e:invert}
L _{q,N} (x)   =  \frac{\mu (q)  } {\phi (q)} c_q (-x)  \cdot   { \gamma _N} \ast  \widecheck { \eta _{s}} (x)
\end{equation}
\end{lemma}

\begin{proof}
For $q\geq 2$, compute 
\begin{align*}
L _{q,N} (x) & = \int _{\mathbb T } \widehat {L _{q,N}} (\xi ) e (-x \xi ) \; d \xi 
\\
& =  \frac{\mu (q)  } {\phi (q)} \sum_{a\in \mathbb A _q}  \int _{\mathbb T } \widehat { \gamma_N } \cdot \eta_s  ( \xi -a/q) e (-x \xi ) \; d \xi 
\\
& =  \frac{\mu (q)  } {\phi (q)} \gamma _N  \ast \widecheck {\eta _{s}} (x)\sum_{a\in \mathbb A _q} e _q (-ax) 
 =  \frac{\mu (q)  } {\phi (q)}  c_q (-x) \cdot \gamma _N  \ast  \widecheck {\eta _{s}} (x),
\end{align*}
where we are using the notation of Ramanujan sums \eqref{e:cq}.  
Above $\widecheck{\eta_s}$ is understood as $\widecheck{\eta_{s,\mathrm{per}}}$, where $\eta_{s,\mathrm{per}}$ is the $1$-periodic extension of $\eta_s$.
For $q=1$, \eqref{e:invert} holds since $c_1(x)\equiv 1$.
\end{proof}

The term on the right in \eqref{e:invert} includes an average $\gamma_N$. It also includes a Ramanujan sum term. 
One should note that $ c _{q} (0) = \phi (q)$, but this is far from typical behavior. 
This crude estimate  shows that for most $ x$, $ c_q (x) $ is about one.   

\begin{lemma}\label{l:Ram}  
For any $ \epsilon >0$, and integer $ k > 1$, uniformly in $ M > Q ^{k}$, there holds 
\begin{equation} \label{e:Ram<}
  \Bigl[
   \frac{1}M  \sum_{\lvert  x\rvert < M } \Bigl\lvert \sum_{q=1}^Q  \frac {   c_q (x) }{ \phi (q)  } \Bigr\rvert ^{k}
  \Bigr] ^{1/k}     \lesssim Q^{\epsilon }. 
\end{equation}
The implied constant depends upon $k$ and $\varepsilon$.  
\end{lemma}

\begin{proof}[Sketch of Proof]
We will not give a complete proof.  It follows from work of Bourgain \cite{MR1209299}*{(3.43), page 126} that we have, under the assumptions above, that 
for any integer $P$, 
\begin{equation*}
   \Biggl[ \frac{1}M  \sum_{\lvert  x\rvert < M }\Bigl\lvert  \sum_{ P \leq q < 2P}    \lvert   c_q (x)\rvert  \Bigr\rvert ^{k} \Biggr] ^{1/k}\lesssim P^{1+ \epsilon/2 }, \qquad M  > P ^{k}.  
\end{equation*}
This is given a stand-alone proof in \cite{181012344K}*{Lemma 3.13}.  
Using the well known lower bound $ \phi (q) \gtrsim q ^{1- \epsilon /4}$, we see that
\begin{equation}\label{e:rama_1}
\frac{1}M  \sum_{\lvert  x\rvert < M }\Bigl\lvert  \sum_{ P \leq q < 2P}    \frac{\lvert   c_q (x)\rvert}{\phi(q)}  \Bigr\rvert ^{k}\lesssim P^{3k\varepsilon/4}.
\end{equation}
Finally, let integer $m_0$ be such that $2^{m_0}\leq Q<2^{m_0+1}$. We have
\begin{align*}
\frac{1}{M}  \sum_{\lvert  x\rvert < M }\Bigl\lvert  \sum_{q=1}^Q    \frac{\lvert   c_q (x)\rvert}{\phi(q)}  \Bigr\rvert ^{k}
\leq &\frac{1}{M} \sum_{\lvert  x\rvert < M }\Bigl\lvert  \sum_{m=0}^{m_0}\, \sum_{2^m\leq q< 2^{m+1}}\frac{\lvert   c_q (x)\rvert}{\phi(q)}  \Bigr\rvert ^{k}\\
\leq &(m_0+1)^{k-1}\, \sum_{m=0}^{m_0}\, \frac{1}{M}\sum_{\lvert x\rvert<M} \Bigl\lvert \sum_{2^m\leq q< 2^{m+1}}\frac{\lvert   c_q (x)\rvert}{\phi(q)}  \Bigr\rvert ^{k}\\
\lesssim &(\log Q)^{k-1} \sum_{m=0}^{m_0} 2^{3mk\varepsilon/4}\\
\lesssim & Q^{k\varepsilon},
\end{align*}
where we used \eqref{e:rama_1}. This proves the claimed result. 
\end{proof}

\section{Fixed Scale} 

The fixed scale result has fewer complications than the sparse bound.  We show that for any $ 1< p < 2$, there holds 
\begin{equation}\label{e:FS}
N ^{-1} ( \mathcal{A}_N f, g ) \leq C_p  \langle f \rangle _{2E,p} \langle g \rangle _{E,p}, 
\end{equation}
where $E$ is an interval of length $N$, and the inequality is independent of $ N$.  
Since the condition is open with respect to $ p$, it suffices to consider the case of $p' \in \mathbb{N}$, with $f = \mathbf 1_{F}$ supported on $2E$ and 
$g= \mathbf 1_{G}$ supported on $E$. 
We trivially have 
\begin{equation*}
N ^{-1} ( \mathcal{A}_N f, g ) \lesssim \log N \cdot \langle f \rangle_{2E,1} \langle g \rangle_{E,1} 
\end{equation*}
so that we conclude \eqref{e:FS} if   
\begin{equation}\label{e:small}
 (\log N)( \langle f \rangle_{2E,1} \langle g \rangle_{E,1}) ^{1/p'} \leq 1. 
\end{equation}
We assume that this fails, thus 
\begin{equation}\label{e:fixed_assu}
\min\{ \langle f \rangle_{2E,1}, \langle g\rangle_{E,1}\} > (\log N) ^{-p'}.
\end{equation}

Now, we prove this auxiliary estimate--the High Low estimate.
For constants $ 1\leq J \leq (\log N) ^{p'}$, we can write $ \mathcal{A}_N f = H +L$ where 
\begin{align}\label{e:FSH}
\langle H  \rangle _{E,2}& \lesssim J ^{-1+ \frac{1}{p'}} \langle f \rangle_{2E,1}^{1/2} 
\\ \label{e:FSL}
\langle L \rangle _{E, \infty } & \lesssim J ^{1/p'} \langle f \rangle_{2E,1}^{1/p}. 
\end{align}
The implied constants depend upon $ p$.  The term $ H$ is the High term, and it satisfies a quantified $ \ell ^{2}$ estimate, while 
   $L$ satisfies something close to the $ \ell ^{1} \to \ell ^{\infty }$ endpoint. It consists of the `low frequency' terms.

From this, it follows that 
\begin{equation*}
N ^{-1} (\mathcal{A}_N f, g ) \lesssim J ^{-1+ \frac{1}{p'}}  ( \langle f \rangle_{2E,1} \langle g \rangle_{E,1}  ) ^{1/2} 
+ J ^{1/p'}   \langle f \rangle_{2E,1} ^{1/p} \langle g \rangle_{E,1} . 
\end{equation*}
The two sides are equal provided that 
\begin{equation} \label{e:equal}
J  \simeq   \langle  f\rangle_{2E,1} ^{1/2 - 1/p}  \langle g \rangle_{E,1} ^{-1/2}.
\end{equation}
By our lower bound on $\langle f \rangle_{2E,1}$ and $\langle g\rangle_{E,1}$ from \eqref{e:fixed_assu}, this is an allowed choice of $ J$.
And, then \eqref{e:FS} follows. 

\smallskip 
It remains to prove \eqref{e:FSH} and \eqref{e:FSL}.  Apply our decomposition of the averaging operator  \eqref{e:Approx} with $ A=p'$ and $K=J$.  With the notation from \eqref{e:Approx}, 
set $ H = \mathcal F ^{-1} (r _{N,A,J} \widehat f ) $. 
The $ \ell ^2 $ estimate \eqref{e:FSH} on $ H$ follows from the $ L ^{\infty }$ bound on $ r _{N,A,J}$.  Turning to \eqref{e:FSL}, the estimate for $ L$, 
from \eqref{e:invert}, we have 
\begin{align*}
\Bigl\lvert \sum_{q=1}^J L _{q,N}\ast f  (x) \Bigr\rvert
& \leq 
\sum_{q=1}^J \left( \frac{\lvert  c_q (\cdot )\rvert } {\phi (q)}  \lvert { \gamma _N} \ast \widecheck { \eta _s} (\cdot) \lvert \right) \ast f (x) 
\\
& \leq 
  \sum_{y=1}^{N} \sum_{q=1}^J \frac{\lvert c_{q}(y) \rvert }{ \phi(q) } \lvert \gamma_{N} \ast \widecheck{\eta_{s} }(y) \lvert \,  f(x-y).
\end{align*}  
Here note that
\begin{equation}\label{e:measure}
|\gamma_N\ast \widecheck{\eta_s}(y)|\leq \frac{1}{N} \|\widecheck{\eta_s}\|_{\ell^1}\lesssim \frac{1}{N}.
\end{equation}
Hence
\begin{align*}
\Bigl\lvert \sum_{q=1}^J L _{q,N}\ast f  (x) \Bigr\rvert
& \leq 
  \frac{1}{N}\sum_{y=1}^{2N} \sum_{q=1}^J \frac{ \lvert c_{q}(y)\rvert }{ \phi(q) }  f(x-y)
\\
& \lesssim 
  \Bigl[
 \frac{1}{2N} \sum_{y=1} ^{2N}\Bigl\lvert \sum_{q=1}^J \frac{\lvert  c_q ( y )\rvert  } {\phi (q)  } \Bigr\rvert ^{p'}
\Bigr] ^{1/p'}     
 \Bigl[
 \frac{1}{2N} \sum_{y=1} ^{2N} f (x -y)
\Bigr] ^{1/p}     
\\
& \lesssim J ^{1/p'} \langle f \rangle_{2E,1} ^{1/p}.  
\end{align*}
Above, we have appealed to H\"older inequality and \eqref{e:Ram<}, with appropriate choice of parameters.  Note this \eqref{e:Ram<} only applies for $ N> N _{p}$, 
for a choice of $ N_p$ that is only depending on $ p$.  
After that, we simplify the  expression, since  $ f$ is an indicator set.   This completes the proof.

\section{Sparse Bound} 

We prove the sparse bound in Theorem~\ref{t:sparse}.  The sparse bound is stronger for smaller choices of $ (r,s)$, and so it suffices 
to prove the $(p,p)$ sparse bound for all $ 1< p < 2$.    Again, by openness of the condition we are proving, it suffices to restrict attention 
to functions $ f, g$ that are indicator sets.

The sparse bound is proved by recursion, which depends upon the following definition. 
Let $E$ be an interval of length $2^{n_0}$. 
Let $f=\mathbf 1_F$ be supported on $2E$, and $g=\mathbf 1_G$ be supported on $E$. 
Let $ \tau \;:\; E \to \{2 ^{n} \;:\; 1\leq n \leq n_0\}$ be a choice of stopping time.  
We say that $ \tau $ is \emph{admissible} if for any interval $ I\subset E$ such that $ \langle f \rangle_{3I,1} > 100 \langle f  \rangle_{2E,1}$, there holds 
\begin{equation}\label{e:admiss}
\inf _{x\in I} \tau (x) > \lvert  I\rvert.   
\end{equation}
We will have direct recourse to this at the end of the proof of the Lemma below. 

\begin{lemma}\label{l:admiss}  For all admissible stopping times, and $ 1< p < 2$, there holds 
\begin{equation}\label{e:admiss<}
(\mathcal{A}_{\tau } f ,g) \lesssim (\langle f \rangle_{2E,1} \langle g \rangle_{E,1}) ^{1/p} \lvert  E\rvert.  
\end{equation}
\end{lemma}

It is a routine argument to see that this implies the sparse bound as written in Theorem~\ref{t:sparse}, see \cite{2019arXiv190705734H}*{Lemma 2.8} or  \cite{181002240}*{Lemma 2.1}.  
We prove the Lemma with the auxiliary High Low construction.  For integers $ J = 2 ^{j}$,  we write $ \mathcal{A}_{\tau } f \leq H +L $ where 
\begin{align}\label{e:SpH}
\langle H  \rangle _{E,2} &\lesssim J ^{-1+1/p'}  \langle f \rangle_{2E,1} ^{1/2}, 
\\
\label{e:SpL} 
\langle L  \rangle _{E, \infty  }&  \lesssim J ^{1/p'} \langle f \rangle_{2E,1} ^{1/p}. 
\end{align}
The conclusion of \eqref{e:admiss<} is very similar to the earlier argument in \eqref{e:equal}, and we omit the details.

\smallskip 

We proceed with the construction of the High and Low terms.  We begin with the trivial bound, following from admissibility, 
\begin{equation}  \label{e:Bad}
\langle  \mathcal{A}_{ \tau } f \rangle_{E,\infty} \lesssim \sup _{x} (\log \tau (x)) \langle f \rangle_{2E,1}.  
\end{equation}
On the set $ B = \{ \log \tau (x) \leq  D_p J ^{1/p'}\}$, we see that \eqref{e:SpL} holds.  Here, $D_p$ is a constant that depends only on $p$, which we specify in the discussion of the Low term below.  
We proceed under the assumption that  the set $ B$ is empty.  
Hence the following holds on $E$:
\begin{equation}\label{e:sparse_assu}
\tau(x)\geq D_p J^{1/p'}.
\end{equation}

We are then concerned with averages $ A _{2 ^{n}} f $, where $ n\geq D_p J ^{1/p'}$.  
Let 
\begin{align*}
m:=(p'+1)\lfloor \log_2 n\rfloor.
\end{align*}
Hence $(n/2)^{p'+1}< 2^m\leq n^{p'+1}$.
Apply the decomposition \eqref{e:Approx} 
with $ N= 2 ^{n}$, $A=p'+1$ and $K=2^m$.  Then, we have 
\begin{equation*}
\widehat {\mathcal{A}_ {2 ^{n}}} =  \sum_{q=1}^{2^m} \widehat{L _{2 ^{n},q}} + \rho _{2^n}, 
\end{equation*}
where $ \lVert \rho _{2 ^{n}} \rVert _{\infty } \lesssim n ^{- p'}$.  Our first contribution to the term $ H$ is  $H_1=\lvert \widecheck {\rho _{\tau }} \ast f \rvert$. 
Note that by a familiar square function argument, 
\begin{align*}
\lVert H _{1}\rVert _2 ^2 & \leq \sum_{n \geq D_p J ^{1/p'}}  \lVert   \widecheck  {\rho _{2 ^{n}}} \ast f \rVert_2 ^2 
\\
& \lesssim \lVert f\rVert_2 ^2 \sum_{n \geq D_p J ^{1/p'}}   n ^{-2p'} \lesssim J ^{-2+1/p'} \lVert f\rVert_2 ^2 .  
\end{align*}
This satisfies the requirement in \eqref{e:SpH}.  

We continue with the construction of $ H$.  The second contribution is 
\begin{align*}
H_2 & =  \sup _{n  }  \bigl\lvert  \sum_{2 ^j<q\leq 2^m} L _{ 2 ^{n} ,q}   \ast f  \bigr\rvert
\\
& \leq  \sum_{k=j} ^{m}   \sup _{n  }   \bigl\lvert  \sum_{2 ^{k} < q \leq 2 ^{k+1}}   L _{ 2 ^{n} ,q}   \ast  f  \bigr\rvert
\end{align*}
The point of this last line is that the inequality below is a direct consequence of Bourgain's multi-frequency maximal inequality, 
and the bound $ \phi (q) \gtrsim_{p} q^{1-1/(2p')}$: 
\begin{equation} \label{e:multi}
\Bigl\lVert 
 \sup _{n  }   \bigl\lvert  \sum_{2 ^{k} \leq q < 2 ^{k+1}}   L _{ 2 ^{n} ,q}   \ast f \bigr\rvert
\Bigr\rVert_{\ell^2} 
\lesssim  k   \max_{2 ^{k} \leq q < 2 ^{k+1}} \frac{1} {\phi (q)} \cdot \lVert f\rVert_{\ell^2} \lesssim 2^{-k(1-1/p')} \lVert f\rVert_{\ell^2}. 
\end{equation}
Summing this estimate over $ k\geq j$ completes the analysis of the High term.  

\begin{remark}
 One of the main results of Bourgain \cite{MR1019960} is the multi-frequency maximal inequality, a key aspect of discrete Harmonic Analysis. 
 In the form that we have used it in \eqref{e:multi}, see for instance \cite{2018arXiv180309431K}*{Prop. 5.11}.
\end{remark}

The term that remains is the Low term below. 
We appeal to \eqref{e:invert}, to see that  
\begin{equation}\label{e:sparse1}
\Bigl\lvert  \sum_{q=1}^J L _{q, \tau }  \ast f   (x) \Bigr\rvert\leq 
\sum_{q=1}^J \left( \frac{\lvert  c_q ( \cdot )\rvert } {\phi (q)}  \lvert \gamma _{\tau } \ast \widecheck {\eta _s} ( \cdot ) \rvert \right) \ast f (x)
\end{equation}
We need the following simple Lemma concerning $\gamma_{\tau}\ast \widecheck{\eta_s}$.
\begin{lemma}\label{l:measure}
We have
\begin{align*}
\lvert \gamma_{\tau}\ast \widecheck{\eta_s}(y) \rvert\lesssim
\begin{cases}
\tau^{-1} \qquad\qquad \text{if } |y|\leq 4\tau\\
2^{-2k}\, \tau^{-1}\qquad \text{if } |y|\in (2^k\tau, 2^{k+1}\tau],\ \text{ for } k\geq 2.
\end{cases}
\end{align*}
\end{lemma}
\begin{proof}
The proof for the case $|y|\leq 4\tau$ follows from \eqref{e:measure}. 
Now, assume $|y|\in (2^k\tau, 2^{k+1}\tau]$ for $k\geq 2$.
We have
\begin{align*}
|\gamma_{\tau}\ast \widecheck{\eta_s}(y)|
\leq \frac{1}{8^s \tau}\sum_{z=1}^\tau \lvert \widecheck{\eta}(\frac{y-z}{8^s})\rvert
\lesssim \frac{1}{8^s \tau} \sum_{z=1}^\tau \left(1+(\frac{y-z}{8^s})^2\right)^{-1}
\lesssim \frac{8^s}{2^{2k}\tau^2}
\lesssim \frac{1}{2^{2k}\tau}.
\end{align*}
In the last inequality we used $8^s\leq q^3\leq J^3<2^{D_p J^{1/p'}}<\tau(x)$, due to \eqref{e:sparse_assu} with a proper choice of $D_p$.
\end{proof}

Plugging the estimates in Lemma \ref{l:measure} into \eqref{e:sparse1}, we have
\begin{align*}
\Bigl\lvert  \sum_{q=1}^J L _{q, \tau }  \ast f(x)\Bigr\rvert 
\lesssim&
\frac{1}{\tau} \sum_{|y|\leq 4\tau}\, \sum_{q=1}^J \frac{|c_q(y)|}{\phi(q)} f(x-y)\\
&+ \sum_{k=2}^{\infty} \frac{1}{2^{2k}\tau}\sum_{|y|\leq 2^{k+1}\tau}\, \sum_{q=1}^J \frac{|c_q(y)|}{\phi(q)} f(x-y)\\
\lesssim &
\Bigl[ \sum_{ y=1} ^ {\tau}  \frac{1} {\tau} \Bigl\lvert \sum_{q=1}^J
  \frac{\lvert  c_q ( y )\rvert   } {\phi (q)  }
\Bigr\rvert ^{p'}  \Bigr] 
^{\frac{1}{p'}}    
\Bigl[    \frac{1} {\tau} \sum_{ y=1} ^ {\tau}  f (x-y) \Bigr] ^{\frac{1}{p}} \\
&+\sum_{k=2}^{\infty} \frac{1}{2^k} \Bigl[ \sum_{ y=1} ^ {2^{k+1}\tau}  \frac{1} {2^{k+1}\tau} \Bigl\lvert \sum_{q=1}^J
  \frac{\lvert  c_q ( y )\rvert   } {\phi (q)  }
\Bigr\rvert ^{p'}  \Bigr] 
^{\frac{1}{p'}}   
\Bigl[    \frac{1} {2^{k+1}\tau} \sum_{ y=1} ^ {2^{k+1}\tau}  f (x-y) \Bigr] ^{\frac{1}{p}}\\
\lesssim &J ^{1/p'} \langle f  \rangle _{2E,1} ^{1/p}.  
\end{align*}
We use H\"older's inequality in $ \ell ^{p}$-$\ell ^{p'}$, and use \eqref{e:Ram<} above to gain the factor of $ J ^{1/p'}$.  
Recall that  \eqref{e:Ram<} holds in this setting, since we assumed \eqref{e:sparse_assu}.
Thus, $ J^{p'} < 2^{D_pJ^{1/p'}} <   \tau (x)$, for appropriate choice of constant $ D_p$.  
 Note that admissibility of $ \tau $, namely the condition \eqref{e:admiss},  gives us the estimate in terms of $ \langle f \rangle_{2E,1}$. 
This completes the proof of \eqref{e:SpL}, and completes the proof of the sparse bound. 


\end{document}